\RequirePackage{fix-cm}
\documentclass[smallextended]{svjour3}       
\smartqed  
\usepackage{graphicx}

\usepackage{algorithm}
\usepackage{algorithmic}

\usepackage{subfigure}
\usepackage{float}

\usepackage{amsfonts}

\usepackage{verbatim}

\usepackage{amssymb}

\begin{document}

\title{Sparse Portfolio Selection via Non-convex Fraction Function}


\author{Angang Cui$^{1}$\and
        Jigen Peng$^{1,\ast}$\and
        Chengyi Zhang$^{2}$\and
        Haiyang Li$^{2}$\and
        Meng Wen$^{2}$
}


\institute{$\ast$ Corresponding author\\
            Jigen Peng \at
            \email{jgpengxjtu@126.com;}\\
           1 School of Mathematics and Statistics, Xi'an Jiaotong University, Xi'an, 710049, China\\
           2 School of Science, Xi'an Polytechnic University, Xi'an, 710048, China
}

\date{Received: date / Accepted: date}

\maketitle

\begin{abstract}
In this paper, a continuous and non-convex promoting sparsity fraction function is studied in two sparse portfolio selection models with and without short-selling constraints. Firstly, we study the properties of the optimal solution to the problem $(FP_{a,\lambda,\eta})$ including the first-order and the second optimality condition and the lower and upper bound of the absolute value for its nonzero entries. Secondly, we develop the thresholding representation theory of the problem $(FP_{a,\lambda,\eta})$. Based on it, we prove the existence of the resolvent operator of gradient of $P_{a}(x)$, calculate its analytic expression, and propose an iterative fraction penalty thresholding (IFPT) algorithm to solve the problem $(FP_{a,\lambda,\eta})$. Moreover, we also prove that the value of the regularization parameter $\lambda>0$ can not be chosen too large. Indeed, there exists $\bar{\lambda}>0$ such that the optimal solution to the problem $(FP_{a,\lambda,\eta})$ is equal to zero for any $\lambda>\bar{\lambda}$. At last, inspired by the thresholding representation theory of the problem $(FP_{a,\lambda,\eta})$, we propose an iterative nonnegative fraction penalty thresholding (INFPT) algorithm to solve the problem $(FP_{a,\lambda,\eta}^{\geq})$. Empirical results show that our methods, for some proper $a>0$, perform effective in finding the sparse portfolio weights with and without short-selling constraints.

\keywords{Markowitz mean-variance model\and Sparse portfolio selection\and Short-selling\and Non-convex fraction function\and IFPT algorithm\and INFPT algorithm}
\subclass{65K05\and 90C26\and 90C90}
\end{abstract}

\section{Introduction}\label{section1}
The classical mean-variance (M-V) portfolio selection model \cite{mark1}, also known as Markowitz mean-variance model constructed in a frictionless world, has been widely used in
economic modeling of finance markets and asset pricing. In M-V model, the return and the risk of a portfolio are measured by the mean and the variance of the portfolio random returns,
respectively, and it aims to find the optimal asset weight vector that minimizes the portfolio variance, subject to the constraint that the portfolio exhibits a desired portfolio
return. It means investors need to invest in a large number of assets. The M-V portfolio theory believes that using diversified portfolio investment can effectively control the
portfolio risk. When the number of assets is typically large, it means investors need to invest in a large number of assets and the solution of the M-V model is usually non-zero on
almost all of the components. However, the large number of assets always lead to the high transaction costs and complexity of portfolio management and the M-V model becomes
numerically unstable \cite{bro8}. Therefore, almost all the investors can only invest in a limited number of assets.

The number restriction on assets motivates many researchers to study the sparse M-V portfolio selection problem, that is, get a sparse asset allocation (solution) with better
out-of-sample performances and to reduce the transaction costs and the complexity of portfolio management. This sparse problem is often called cardinality constrained portfolio
optimization, and some variations thereof, have been fairly intensively studied in \cite{arma2,ber3,bien4,chang5,crama6,gfiel7,bro8,got9,yena10,lore11,teng12}. For the sake of
uniformity, in this paper, we call it the sparse portfolio selection problem. Unfortunately, this sparse problem, motivated by the need of inducing sparsity on the selected portfolio
to reduce transaction costs, complexity of portfolio management, and instability of the solution is a difficult, in fact NP-hard, combinatorial problem (see \cite{bien4}).
In \cite{bro8}, the $\ell_{1}$-norm is proposed to promote the sparsity of assets in the portfolio, as argued by authors, helps inducing sparsity of the selected portfolio and can
be a remedy to the high instability of classic methods for portfolio selection when short-selling is permitted. However, $\ell_{1}$-regularization approach is not effective in
promoting sparsity in presence of budget and no-short-selling constraints \cite{lore11}. Moreover, it tends to lead to biased estimation by shrinking all the entries toward to zero
simultaneously, and sometimes results in over-penalization as the $\ell_{1}$-norm in compressed sensing (see \cite{dau13}).

Inspired by the good performance of the non-convex fraction function in image restoration \cite{gem14}, and based on authors' recent researches on fraction regularization in compressed sensing \cite{li15}, we propose two sparse fraction portfolio selection models with and without short-selling constraints in this paper. The proposed sparse fraction portfolio selection model with short-selling constraint can generate optimal portfolios with better sparsity than the portfolio selection models using $\ell_{1}$-regularization do, and the sparse fraction portfolio selection
model without short-selling constraint can also performs effective in finding the sparse portfolio weights.

The rest of the paper is organized as follows. In Section \ref{section2}, we review some sparse portfolio selection models, and then present two sparse portfolio selection models by introducing fraction regularization on portfolio weights. The Section 3 is devoted to discussing the properties of the optimal solution to the regularization problem $(FP_{a,\lambda,\eta})$ including the first-order and the second optimality condition and the lower and upper bound of the absolute value for its nonzero entries. Moreover, we also proved that the value of the regularization parameter $\lambda$ can not be chosen too large. Indeed, there exists $\bar{\lambda}>0$ such that the optimal solution to the problem $(FP_{a,\lambda,\eta})$ is equal to zero for any $\lambda>\bar{\lambda}$. In Section \ref{section3}, we propose the IFPT algorithm to solve the problem $(FP_{a,\lambda,\eta})$ and, inspired by the thresholding representation of the IFPT algorithm, the INFPT algorithm is given to solve the problem $(FP_{a,\lambda,\eta}^{\geq})$. In Section \ref{section4}, we present the experiments with a series of portfolio selection applications to demonstrate the effectiveness of the new algorithms. We conclude this paper in Section \ref{section5}.

\section{Sparse portfolio selection models}\label{section2}
\subsection{The M-V portfolio selection model}\label{subsection2-1}
Let $r_{t}=(r_{1,t},r_{2,t},\cdots,r_{n,t})^{\top}\in \mathbb{R}^{n}$ be the vector of asset returns at time $t$, $t=1,2,\cdots,T$, $E(r_{t})=\mu$ and $Q=E[(r_{t}-\mu)(r_{t}-\mu)^{\top}]$ be the
mean return vector and the covariance matrix of asset returns, where $r_{it}$ is the return of asset $i$ at time $t$. The traditional Markowitz portfolio selection model (see \cite{mark1}) can be expressed as follows
\begin{equation}\label{equ1}
\begin{array}{llll}
&&\displaystyle\min_{x\in \mathbb{R}^{n}}\ \ \ \ \sigma:=x^{\top}Qx\\
&&\mathrm{\ s.t.}\ \ \ \ \ \ \mu^{\top}x=\beta\\
&&\ \ \ \ \ \ \ \ \ \ \ e_{n}^{\top}x=1,
\end{array}
\end{equation}
where $x=(x_{1},x_{2},\cdots,x_{n})^{\top}\in \mathbb{R}^{n}$ is the vector of asset weights, $e_{n}\in \mathbb{R}^{n}$ is the vector of all ones, $\beta$ is the minimum expected return from the portfolio that is expected by an investor. Note also that, if the non-short-selling (without short-selling) constraint $x\geq0$ is added to problem (\ref{equ1}), we can recast this problem
as portfolio selection model without short-selling constraint
\begin{equation}\label{equ100}
\begin{array}{llll}
&&\displaystyle\min_{x\in \mathbb{R}^{n}}\ \ \ \ \sigma:=x^{\top}Qx\\
&&\mathrm{\ s.t.}\ \ \ \ \ \ \mu^{\top}x=\beta\\
&&\ \ \ \ \ \ \ \ \ \ \ e_{n}^{\top}x=1\\
&&\ \ \ \ \ \ \ \ \ \ \ x\geq0.
\end{array}
\end{equation}

Since $Q=E[(r_{t}-\mu)(r_{t}-\mu)^{\top}]$, we have
$$x^{\top}Qx=E[|\beta-x^{\top}r_{t}|^{2}]=\frac{1}{T}\|Rx-\beta e_{T}\|_{2}^{2}$$
where $R=(r_{1},r_{2},\cdots,r_{T})^{\top}\in \mathbb{R}^{T\times n}$. Then the Markowitz portfolio selection model (\ref{equ1}) can be expressed as follows
\begin{equation}\label{equ2}
\begin{array}{llll}
&&\displaystyle\min_{x\in \mathbb{R}^{n}}\ \ \ \ \frac{1}{T}\|Rx-\beta e_{T}\|_{2}^{2}\\
&&\mathrm{\ s.t.}\ \ \ \ \ \ \mu^{\top}x=\beta\\
&&\ \ \ \ \ \ \ \  \ \ \ e_{n}^{\top}x=1.
\end{array}
\end{equation}

Usually, both the vector $\mu$ and the matrix $Q$ are not known analytically but can be estimated using historical data.

\subsection{The sparse portfolio selection model}\label{subsection2-2}
The sparsity requirement comes from the  real world practice, where the administration of a portfolio made up of a large number of assets, possibly with very small holdings for some of them, is clearly not desirable because of the transactions costs and the complexity of management. The sparsity restricted model is often called cardinality constrained portfolio selection problem by limiting the number of assets in the portfolio, and defined as follows

\begin{equation}\label{equ3}
\begin{array}{llll}
&&\displaystyle\min_{x\in \mathbb{R}^{n}}\ \ \ \ \frac{1}{T}\|Rx-\beta e_{T}\|_{2}^{2}\\
&&\mathrm{\ s.t.}\ \ \ \ \ \ \mu^{\top}x=\beta\\
&&\ \ \ \ \ \ \ \  \ \ \ e_{n}^{\top}x=1\\
&&\ \ \ \ \ \ \ \  \ \ \ \|x\|_{0}\leq k,
\end{array}
\end{equation}
where the $\|x\|_{0}$ is the $\ell_{0}$-norm of $x$ indicates the number of nonzero components of $x$, and the parameter $k$ is the chosen limit of assets to be held in the portfolio.
For the sake of uniformity, we call it the sparse portfolio selection model in this paper. Unfortunately, sparse problem (\ref{equ3}) motivated
by the need of inducing sparsity on the selected portfolio is a difficult, in fact NP-hard, combinatorial problem (see \cite{bien4}).

In \cite{bro8}, an important sparse portfolio selection model, based on $\ell_{1}$-norm regularization, is proposed to promote the sparsity of assets in the portfolio, as argued by authors, helps
induce sparsity of the selected portfolio and can be a remedy to the high instability of classic methods for portfolio selection when short-selling is permitted. This $\ell_{1}$-norm
regularization portfolio selection model can be viewed as the following mathematical form

\begin{equation}\label{equ4}
\begin{array}{llll}
&&\displaystyle\min_{x\in \mathbb{R}^{n}}\ \ \ \ \frac{1}{T}\|Rx-\beta e_{T}\|_{2}^{2}+\lambda\|x\|_{1}\\
&&\mathrm{\ s.t.}\ \ \ \ \ \ \mu^{\top}x=\beta\\
&&\ \ \ \ \ \ \ \ \ \ \ e_{n}^{\top}x=1,
\end{array}
\end{equation}
where $\lambda>0$ is a regularization parameter, and $\|x\|_{1}=\sum_{i=1}^{n}|x_{i}|$. Problem (\ref{equ4}) is called sparse and stable M-V portfolio selection model in \cite{bro8},
however, sparsity of the resulting portfolio is not guaranteed from problem (\ref{equ4}), since the $\ell_{1}$-norm of the asset weights will result in a constant value of one when asset
weights is nonnegative. Moreover, it tends to lead to biased estimation by shrinking all the entries toward to zero simultaneously, and sometimes results in over-penalization as the
$\ell_{1}$-norm in compressed sensing (see \cite{dau13}).


\subsection{The new sparse portfolio selection model}\label{subsection2-3}
Inspired by the good performance of the non-convex fraction function in image restoration and compressed sensing (see, e.g., \cite{gem14,li15}), we take the non-convex function
$P_{a}(x)$ to substitute the $\ell_{1}$-norm in problem (\ref{equ4}).

The function $P_{a}(x)$ is defined as
\begin{equation}\label{equ5}
P_{a}(x)=\sum_{i=1}^{n}\rho_{a}(x_{i}),\ \ \ a>0
\end{equation}
where
\begin{equation}\label{equ6}
\rho_{a}(t)=\frac{a|t|}{a|t|+1}
\end{equation}
is the fraction function, and the parameter $a\in(0,+\infty)$. It is easy to verify that $\rho_{a}(t)$ is increasing and concave in $t\in[0,+\infty)$. With the change of
parameter $a$, we have
\begin{equation}\label{equ7}
\lim_{a\rightarrow+\infty}\rho_{a}(t)=\left\{
    \begin{array}{ll}
      0, & {\ \ \mathrm{if} \ t=0;} \\
      1, & {\ \ \mathrm{if} \ t\neq 0,}
    \end{array}
  \right.
\end{equation}
and the non-convex function $P_{a}(x)$ interpolates the $\ell_{0}$-norm of vector $x$:
\begin{equation}\label{equ8}
\lim_{a\rightarrow+\infty}P_{a}(x)=\lim_{a\rightarrow+\infty}\sum_{x_{i}\neq 0}\rho_{a}(x_{i})=\|x\|_{0}.
\end{equation}
\begin{figure}[h!]
 \centering
 \includegraphics[width=0.7\textwidth]{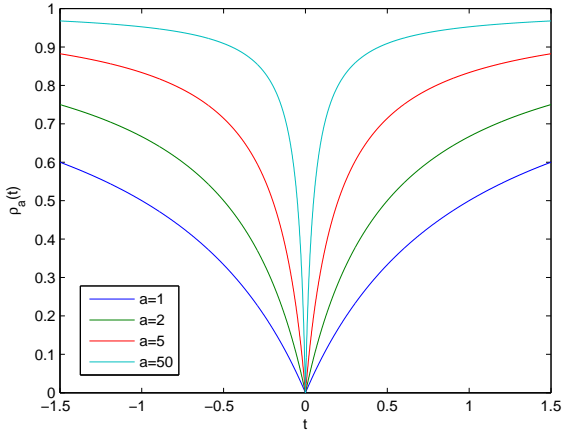}
\caption{The behavior of the fraction function $p_{a}(t)$ for various values of $a$.}
\label{fig:1}
\end{figure}

By this substitution, we translate the problem (\ref{equ4}) with short-selling constraint into the following problem
\begin{equation}\label{equ9}
\begin{array}{llll}
&&\displaystyle\min_{x\in \mathbb{R}^{n}}\ \ \ \ \frac{1}{T}\|Rx-\beta e_{T}\|_{2}^{2}+\lambda P_{a}(x)\\
&&\mathrm{\ s.t.}\ \ \ \ \ \ \mu^{\top}x=\beta\\
&&\ \ \ \ \ \ \ \ \ \ \ e_{n}^{\top}x=1.
\end{array}
\end{equation}
If short-selling is not permitted (without short-selling constraint), the model for the problem (\ref{equ9}) has the form
\begin{equation}\label{equ10}
\begin{array}{llll}
&&\displaystyle\min_{x\in \mathbb{R}^{n}}\ \ \ \ \frac{1}{T}\|Rx-\beta e_{T}\|_{2}^{2}+\lambda P_{a}(x)\\
&&\mathrm{\ s.t.}\ \ \ \ \ \ \mu^{\top}x=\beta\\
&&\ \ \ \ \ \ \ \ \ \ \ e_{n}^{\top}x=1\\
&&\ \ \ \ \ \ \ \ \ \ \ x\geq0.
\end{array}
\end{equation}
Meanwhile, problems (\ref{equ9}) and (\ref{equ10}) could be expressed in the matrix-vector form
\begin{equation}\label{equ11}
\begin{array}{llll}
&&(FP_{a,\lambda})\ \ \ \ \ \displaystyle\min_{x\in \mathbb{R}^{n}}\ \ \ \ \frac{1}{T}\|Rx-\beta e_{T}\|_{2}^{2}+\lambda P_{a}(x)\\
&&\ \ \ \ \ \ \ \ \ \ \ \ \ \ \ \mathrm{\ s.t.}\ \ \ \ \ \ Ax=b
\end{array}
\end{equation}
and
\begin{equation}\label{equ12}
\begin{array}{llll}
&&(FP_{a,\lambda}^{\geq})\ \ \ \ \ \displaystyle\min_{x\in \mathbb{R}^{n}}\ \ \ \ \frac{1}{T}\|Rx-\beta e_{T}\|_{2}^{2}+\lambda P_{a}(x)\\
&&\ \ \ \ \ \ \ \ \ \ \ \ \ \ \ \ \mathrm{\ s.t.}\ \ \ \ \ \ Ax=b\\
&&\ \ \ \ \ \ \ \ \ \ \ \ \ \ \ \ \ \ \ \ \ \ \ \ \ \ \ \ x\geq0,
\end{array}
\end{equation}
where $b=(\beta,1)^{\top}\in \mathbb{R}^{2}$, $A=(\mu,e_{n})^{\top}\in \mathbb{R}^{2\times n}$.

The penalty function problems for $(FP_{a,\lambda})$ and $(FP_{a,\lambda}^{\geq})$ are given by
\begin{equation}\label{equ13}
(FP_{a,\lambda,\eta})\ \ \ \ \ \displaystyle\min_{x\in \mathbb{R}^{n}}\ \ \ \ \frac{1}{T}\|Rx-\beta e_{T}\|_{2}^{2}+\lambda P_{a}(x)+\eta\|Ax-b\|_{2}^{2}
\end{equation}
and
\begin{equation}\label{equ14}
(FP_{a,\lambda,\eta}^{\geq})\ \ \ \ \ \displaystyle\min_{\mathbb{R}^{n}\ni x\geq0}\ \ \ \ \frac{1}{T}\|Rx-\beta e_{T}\|_{2}^{2}+\lambda P_{a}(x)+\eta\|Ax-b\|_{2}^{2}
\end{equation}
where $\eta>0$ is the penalty parameter.

\section{Properties of the problem $(FP_{a,\lambda,\eta})$}\label{section3}
In this section, we discuss some properties of the problem $(FP_{a,\lambda,\eta})$ including the first-order and the second optimality condition and the lower and upper bounds of the
absolute value for its nonzero entries. Moreover, we also prove that the value of the regularization parameter $\lambda$ can not be chosen too large. Indeed, there exists $\bar{\lambda}>0$ such that the optimal solution to the problem $(FP_{a,\lambda,\eta})$ is equal to zero for any $\lambda>\bar{\lambda}$.

\subsection{Lower and upper bounds of the optimal solution}\label{section3-1}

\begin{theorem}\label{the1}
(The first-order optimality condition) Let $x^{\ast}$ be any solution to the problem $(FP_{a,\lambda,\eta})$ and for any $h\in \mathbb{R}^{n}$ with $\mathrm{supp}(h)\subseteq \mathrm{supp}(x^{\ast})$
($\mathrm{supp}(x^{\ast})$ represents the support of vector $x^{\ast}$, $\sharp=\{x_{i}| x_{i}\neq 0\}$),
\begin{equation}\label{equ15}
\frac{2}{T}\langle \beta e_{T}-Rx^{*},Rh\rangle+2\eta\langle b-Ax^{\ast},Ah\rangle=\lambda\sum_{i\in \mathrm{supp}(x^{\ast})}\frac{ah_{i}\mathrm{sign}(x_{i}^{\ast})}{(1+a|x_{i}^{\ast}|)^{2}}.
\end{equation}
\end{theorem}

\begin{proof}
Let $x^{*}$ be any solution to the problem $(FP_{a,\lambda,\eta})$. Then, for all $\tau\in \mathbb{R}$ and $h\in \mathbb{R}^{n}$, the following inequality holds
\begin{eqnarray*}
 &&\frac{1}{T}\|Rx^{\ast}-\beta e_{T}\|_{2}^{2}+\lambda P_{a}(x^{*})+\eta\| Ax^{*}-b\|_{2}^{2}\\
 &&\leq\frac{1}{T}\|R(x^{\ast}+\tau h)-\beta e_{T}\|_{2}^{2}+\lambda P_{a}(x^{*}+\tau h)+\eta\| A(x^{*}+\tau h)-b\|_{2}^{2},
\end{eqnarray*}
equivalently,
\begin{equation}\label{equ16}
\begin{array}{llll}
&&\displaystyle\frac{\tau^{2}}{T}\|Rh\|_{2}^{2}+\eta \tau^{2}\|Ah\|_{2}^{2}+\frac{2\tau}{T} \langle Rx^{*}-\beta e_{T},Rh\rangle\\
&&+2\eta \tau\langle Ax^{*}-b,Ah\rangle+\lambda(P_{a}(x^{*}+\tau h)-P_a(x^{*}))\geq 0.
\end{array}
\end{equation}
If $\mathrm{supp}(h)\subseteq \mathrm{supp}(x^{*})$, then for a small enough $\tau$ the vector $x^{\ast}$, $x^{\ast}+\tau h$ and $x^{\ast}-\tau h$ have the same sign, and
$$P_{a}(x^{*}+\tau h)-P_{a}(x^{*})=\sum_{i\in \mathrm{supp}(x^{*})}\Big(\frac{a|x^{*}_{i}+\tau h_{i}|}{1+a|x^{*}_{i}+\tau h_{i}|}-\frac{a|x^{*}_{i}|}{1+a|x^{*}_{i}|}\Big).$$
Dividing by $\tau>0$ both sides of the inequality (\ref{equ17}) and letting $\tau\rightarrow 0$ yield
\begin{equation}\label{equ17}
\frac{2}{T}\langle Rx^{*}-\beta e_{T},Rh\rangle+2\eta\langle Ax^{*}-b,Ah\rangle+\lambda\sum_{i\in \mathrm{supp}(x^{*})}\frac{ah_{i}\mathrm{sign}(x_{i}^{\ast})}{(1+a|x^{*}_{i}|)^2}\geq 0.
\end{equation}
Obviously, the above inequality also holds for $-h$ which leads to the equality (\ref{equ15}).
This completes the proof.
\end{proof}

Choosing $h$ as the $i^{th}$ base vector $e_{i}$ for each $i=1,2,\cdots,n$ in (\ref{equ15}), we can derive the following corollary.

\begin{corollary}\label{cor1}
Suppose that $x^{*}$ is the solution of the problem $(FP_{a,\lambda,\eta})$. Then, for $i\in \mathrm{supp}(x^{*})$,
\begin{equation}\label{equ18}
\frac{2}{T}(R^T(\beta e_{T}-Rx^{*}))_{i}+2\eta(A^T(b-Ax^{*}))_{i}=\frac{a\lambda}{(1+a|x^{*}_{i}|)^{2}}.
\end{equation}
\end{corollary}

Following the analysis adopted above, we can further establish the following optimality condition.

\begin{theorem}\label{the2}
(The second-order optimality condition) Every solution $x^{*}$ to the problem $(FP_{a,\lambda,\eta})$ satisfies the following condition:\\\\
(1) For all $h\in \mathbb{R}^{n}$ with $\mathrm{supp}(h)\subseteq \mathrm{supp}(x^{*})$,
\begin{equation}\label{equ19}
\frac{1}{T}\|Rh\|_{2}^{2}+\eta\|Ah\|_{2}^{2}\geq\lambda\sum_{i\in \mathrm{supp}(x^{*})}\frac{2a^2h_{i}^{2}}{(1+a|x^{*}_{i}|)^{3}}.
\end{equation}
(2) Moreover, for all $i\in \mathrm{supp}(x^{*})$ and $a>\frac{\frac{1}{T}\|R_{i}\|_{2}^{2}+\eta\|A_{i}\|_{2}^{2}}{\sqrt{\lambda}}$, it holds that
\begin{equation}\label{equ20}
|x_i^{*}|\geq\frac{\sqrt{\lambda}}{\frac{1}{T}\|R_{i}\|_{2}^{2}+\eta\|A_{i}\|_{2}^{2}}-\frac{1}{a}.
\end{equation}
\end{theorem}

\begin{proof}
(1) Let $\mathrm{supp}(h)\subseteq \mathrm{supp}(x^{*})$. Then, incorporating the equality (\ref{equ15}) into the inequality (\ref{equ17}) yields that, for all $\tau\in \mathbb{R}$,
$$\frac{\tau^{2}}{T}\|Rh\|_{2}^{2}+\eta \tau^{2}\|Ah\|_{2}^{2}-\lambda\sum_{i\in \mathrm{supp}(x^{*})}\frac{\tau ah_{i}\mathrm{sign}(x_{i}^{\ast})}{(1+a|x^{*}_{i}|)^{2}}+\lambda(P_{a}(x^{*}+\tau h)-P_{a}(x^{*}))\geq 0,$$
or equivalently
\begin{equation}\label{equ21}
\begin{array}{llll}
&&\displaystyle\frac{1}{T}\|Rh\|_{2}^{2}+\eta\|Ah\|_{2}^{2}\\
&&\geq\displaystyle\frac{\lambda}{\tau^{2}}\Big(\sum_{i\in \mathrm{supp}(x^{*})}\frac{\tau ah_{i}\mathrm{sign}(x_{i}^{\ast})}{(1+a|x^{*}_{i}|)^{2}}-(P_{a}(x^{*}+\tau h)-P_{a}(x^{*}))\Big).
\end{array}
\end{equation}
Hence, letting $\tau\rightarrow0$ on the right-hand of inequality above, we have the inequality (\ref{equ19}).

(2) If we replace $h$ in inequality (\ref{equ21}) with the base vector $e_{i}$ for every $i\in \mathrm{supp}(x^{\ast})$, then we have the component-wise inequality
$$\frac{1}{T}\|R_{:,i}\|_{2}^{2}+\eta\|A_{:,i}\|_{2}^{2}\geq\frac{\lambda}{\tau^{2}}\bigg(\frac{at\mathrm{sign}(x_{i}^{\ast})}{(1+a|x_{i}^{\ast}|)^{2}}-
\frac{a|x_{i}^{\ast}+\tau|}{1+a|x_{i}^{\ast}+\tau|}+\frac{a|x_{i}^{\ast}|}{1+a|x_{i}^{\ast}|}\bigg)$$
where $R_{:,i}$ and $A_{:,i}$ represent the $i$-th column of the matrix $R$ and $A$, respectively.
Particularly, above inequality is available for $\tau=-x_{i}^{*}$. So, we have
$$\frac{1}{T}\|R_{:,i}\|_{2}^{2}+\eta\|A_{:,i}\|_{2}^{2}\geq\frac{\lambda}{(x_{i}^{\ast})^{2}}\bigg(-\frac{a|x_{i}^{\ast}|}{(1+a|x_{i}^{\ast}|)^{2}}
+\frac{a|x_{i}^{\ast}|}{1+a|x_{i}^{\ast}|}\bigg)$$
It follows that
$$\frac{1}{T}\|R_{i}\|_{2}^{2}+\eta\|A_{i}\|_{2}^{2}\geq\frac{\lambda a^{2}}{(1+a|x^{*}_{i}|)^{2}}.$$
From the inequality above, the inequality (\ref{equ20}) immediately follows. This completes the proof.
\end{proof}

\begin{theorem}\label{the3}
Suppose that  $x^{*}$ is the optimal solution to the problem $(FP_{a,\lambda,\eta})$. If $\lambda>\frac{1}{T}\|\beta e\|_{2}^{2}+\eta\| b\|_{2}^{2}$, then
\begin{equation}\label{equ22}
\| x^{*}\|_{\infty}\leq \frac{\frac{1}{T}\|\beta e_{T}\|_{2}^{2}+\eta\| b\|_{2}^{2}}{a(\lambda-(\frac{1}{T}\|\beta e_{T}\|_{2}^{2}+\eta\| b\|_{2}^{2}))}.
\end{equation}
\end{theorem}

\begin{proof}
Let $x^{*}$ be the optimal solution to the problem $(FP_{a,\lambda,\eta})$. Then we have
$$f(x^{\ast})=\frac{1}{T}\|Rx^{\ast}-\beta e_{T}\|_{2}^{2}+\lambda P_{a}(x^{\ast})+\eta\|Ax^{\ast}-b\|_{2}^{2}\leq f(0)=\frac{1}{T}\|\beta e_{T}\|_{2}^{2}+\eta\| b\|_{2}^{2}.$$
Hence $\lambda P_{a}(x^{*})\leq \frac{1}{T}\|\beta e_{T}\|_{2}^{2}+\eta\| b\|_{2}^{2}$, which implies that
$$\frac{a\|x^{*}\|_{\infty}}{1+a\| x^{*}\|_{\infty}}\leq\frac{\frac{1}{T}\|\beta e_{T}\|_{2}^{2}+\eta\| b\|_{2}^{2}}{\lambda}.$$
If $\lambda>\frac{1}{T}\|\beta e_{T}\|_{2}^{2}+\eta\| b\|_{2}^{2}$, then
$$\| x^{*}\|_{\infty}\leq \frac{\frac{1}{T}\|\beta e_{T}\|_{2}^{2}+\eta\| b\|_{2}^{2}}{a[\lambda-(\frac{1}{T}\|\beta e_{T}\|_{2}^{2}+\eta\| b\|_{2}^{2})]}.$$
This completes the proof.
\end{proof}

\subsection{Large regularization parameter $\lambda$ leads to zero solution}\label{section3-2}
Before we embark to this discussion, we should declare that the results derived in this following discussion are worst-case ones, implying that the kind of guarantees we obtained are over-pessimistic for all possibilities.

\begin{lemma} \label{lem1}
Let $x^{\ast}$ of sparsity $r$ be the optimal solution of the problem $(FP_{a,\lambda,\eta})$, the matrix $\tilde{R}$ be the submatrix of $R$ corresponding to $\mathrm{supp}(x^{*})$ and the matrix $\tilde{A}$ be the submatrix of $A$ corresponding to $\mathrm{supp}(x^{*})$. Then the matrix $\frac{2}{T}\tilde{R}^{\top}\tilde{R}+2\eta \tilde{A}^{\top}\tilde{A}$ is positive definite.
\end{lemma}

\begin{proof}
Without loss of generality, we assume
$$x^{\ast}=(x_{1}^{\ast}, x_{2}^{\ast}, \cdots, x_{r}^{\ast}, 0,\cdots, 0)^{\top}.$$
Let $z^{\ast}=(x_{1}^{\ast}, x_{2}^{\ast}, \cdots, x_{r}^{\ast})^{\top}\in \mathbb{R}^{r}$, $\tilde{R}\in \mathbb{R}^{T\times r}$ be the sub-matrix of $R$, whose columns in matrix $R$ corresponding to $z^{\ast}$, and $\tilde{A}\in \mathbb{R}^{2\times r}$ be the sub-matrix of $A$, whose columns in matrix $A$ corresponding to $z^{\ast}$. Define a function $g: \mathbb{R}^{r}\mapsto \mathbb{R}$ by
\begin{equation}\label{equ23}
g(z^{\ast})=\frac{1}{T}\|\tilde{R}z^{\ast}-\beta e_{T}\|_{2}^{2}+\lambda P_{a}(z^{\ast})+\eta\|\tilde{A}z^{\ast}-b\|_{2}^{2}.
\end{equation}
We have
\begin{equation}\label{equ24}
\begin{array}{llll}
f(X^{\ast})&=&\displaystyle\frac{1}{T}\|Rx^{\ast}-\beta e_{T}\|_{2}^{2}+\lambda P_{a}(x^{\ast})+\eta\|Ax^{\ast}-b\|_{2}^{2}\\
&=&\displaystyle\frac{1}{T}\|\tilde{R}z^{\ast}-\beta e_{T}\|_{2}^{2}+\lambda P_{a}(z^{\ast})+\eta\|\tilde{A}z^{\ast}-b\|_{2}^{2}\\
&=&\mathrm{g}(z^{\ast}).
\end{array}
\end{equation}
Since function $g$ is continuously differentiable at $z^{\ast}$. Moreover, in a neighborhood of $z^{\ast}$,
\begin{equation}\label{equ25}
\begin{array}{llll}
g(z^{\ast})=f(x^{\ast})&\leq&\displaystyle\min_{x\in \mathbb{R}^{n}}\Big\{f(x)|x_{i}=0,i=r+1,r+2,\cdots,n\Big\}\\
&=&\displaystyle\min_{z\in \mathbb{R}^{r}}g(z),
\end{array}
\end{equation}
which implies that $z^{\ast}$ is a local minimizer of the function $g$. Hence, the second order necessary condition for
$$\min_{z\in \mathbb{R}^{r}}g(z)$$
holds at $z^{\ast}$.
The second order necessary condition at $z^{\ast}$ gives that the matrix
$$\frac{2}{T}\tilde{R}^{\top}\tilde{R}+2\eta \tilde{A}^{\top}\tilde{A}-\mathrm{Diag}\bigg(\frac{2\lambda a^{2}}{(a|z_{i}^{\ast}|+1)^{3}}\bigg), \ \ \ i=1, 2, \cdots, r$$
is positive semi-definite, and the matrix
$$M=\mathrm{Diag}\bigg(\frac{2\lambda a^{2}}{(a|z_{i}^{\ast}|+1)^{3}}\bigg)$$
is positive. Therefore, the matrix $\frac{2}{T}\tilde{R}^{\top}\tilde{R}+2\eta \tilde{A}^{\top}\tilde{A}$ must be positive definite. This completes the proof.
\end{proof}

Nextly, we shall show that the value of the regularization parameter $\lambda$ of the problem $(FP_{a,\lambda,\eta})$ can not be chosen too large.
\begin{theorem} \label{th7}
Let
\begin{eqnarray*}
\bar{\lambda}&=&\frac{1}{T}\|\beta e_{T}\|_{2}^{2}+\eta\|b\|_{2}^{2}+\frac{1}{a}\Big\|\frac{\beta}{m}R^{\top}e_{T}+\eta A^{\top}b\Big\|_{\infty}\\
&&+\frac{1}{a}\sqrt{\Big\|\frac{\beta}{T}R^{\top}e_{T}+\eta A^{\top}b\Big\|^{2}_{\infty}+2a\Big(\frac{1}{T}\|\beta e_{T}\|_{2}^{2}+\eta\|b\|_{2}^{2}\Big)\Big\|\frac{\beta}{T}R^{\top}e_{T}+\eta A^{\top}b\Big\|_{\infty}}.
\end{eqnarray*}
Then for all $\lambda\geq\bar{\lambda}$, the problem $(FP_{a,\lambda,\eta})$ admits the zero solution.
\end{theorem}
\begin{proof}
By the proof of Lemma \ref{lem1}, the first order necessary condition for
$$\min_{z\in \mathbb{R}^{r}}\mathrm{g}(z)$$
at $z^{\ast}$ gives
\begin{equation}\label{equ26}
\frac{2}{T}\tilde{R}^{\top}(\tilde{R}z^{\ast}-\beta e_{T})+2\eta \tilde{A}^{\top}(\tilde{A}z^{\ast}-b)
+\mathrm{Diag}(\mathrm{sign}(z^{\ast}))\frac{\lambda a}{(a|z^{\ast}|+1)^{2}}=0.
\end{equation}
Multiplying by $(z^{\ast})^{\top}$ both sides of equality above yield
\begin{eqnarray*}
&&(z^{\ast})^{\top}\Big(\frac{2}{T}\tilde{R}^{\top}\tilde{R}+2\eta \tilde{A}^{\top}\tilde{A}\Big)z^{\ast}
-(z^{\ast})^{\top}\Big(\frac{2\beta}{T}\tilde{R}^{\top}e_{T}+2\eta \tilde{A}^{\top}b\Big)\\
&&+(z^{\ast})^{\top}\mathrm{Diag}(\mathrm{sign}(z^{\ast}))\frac{\lambda a}{(a|z^{\ast}|+1)^{2}}=0.
\end{eqnarray*}
Because the matrix $\frac{2}{T}\tilde{R}^{\top}\tilde{R}+2\eta \tilde{A}^{\top}\tilde{A}$ is positive definite (see Lemma \ref{lem1}), and hence
$$-(z^{\ast})^{\top}\Big(\frac{2\beta}{T}\tilde{R}^{\top}e_{T}+2\eta \tilde{A}^{\top}b\Big)+(z^{\ast})^{\top}\mathrm{Diag}(\mathrm{sign}(z^{\ast}))\frac{\lambda a}{(a|z^{\ast}|+1)^{2}}<0,$$
equivalently,
\begin{equation}\label{equ27}
\sum_{i=1}^{r}\bigg(\frac{\lambda a|z_{i}^{\ast}|}{(a|z_{i}^{\ast}|+1)^{2}}-2\Big(\frac{\beta}{T}\tilde{R}^{\top}e_{T}+\eta \tilde{A}^{\top}b\Big)_{i}z_{i}^{\ast}\bigg)<0.
\end{equation}
Since
\begin{eqnarray*}
\lambda&>&\frac{1}{T}\|\beta e_{T}\|_{2}^{2}+\eta\|b\|_{2}^{2}+\frac{1}{a}\Big\|\frac{\beta}{T}R^{\top}e_{T}+\eta A^{\top}b\Big\|_{\infty}\\
&&+\frac{1}{a}\sqrt{\Big\|\frac{\beta}{T}R^{\top}e_{T}+\eta A^{\top}b\Big\|^{2}_{\infty}+2a\Big(\frac{1}{T}\|\beta e_{T}\|_{2}^{2}+\eta\|b\|_{2}^{2}\Big)\Big\|\frac{\beta}{T}R^{\top}e_{T}+\eta A^{\top}b\Big\|_{\infty}},
\end{eqnarray*}
we obtain
\begin{equation}\label{equ28}
\begin{array}{llll}
&&a\lambda^{2}-2\bigg[a\Big(\frac{1}{T}\|\beta e_{T}\|_{2}^{2}+\eta\|b\|_{2}^{2}\Big)+\Big\|\frac{\beta}{T}R^{\top}e_{T}+\eta A^{\top}b\Big\|_{\infty}\bigg]\lambda\\
&&+a\Big(\frac{1}{T}\|\beta e_{T}\|_{2}^{2}+\eta\|b\|_{2}^{2}\Big)^{2}\geq0,
\end{array}
\end{equation}
which implies that
\begin{equation}\label{equ29}
\frac{a[\lambda-(\frac{1}{T}\|\beta e_{T}\|_{2}^{2}+\eta\|b\|_{2}^{2})]^{2}}{\lambda}\geq2\Big\|\frac{\beta}{T}R^{\top}e_{T}+\eta A^{\top}b\Big\|_{\infty}.
\end{equation}
Together with
$$\frac{\lambda a}{(a|z_{i}^{\ast}|+1)^{2}}\geq\frac{a[\lambda-(\frac{1}{T}\|\beta e_{T}\|_{2}^{2}+\eta\|b\|_{2}^{2})]^{2}}{\lambda}$$
and
\begin{equation}\label{equ30}
\Big|\Big(\frac{\beta}{T}\tilde{R}^{\top}e_{T}+\eta \tilde{A}^{\top}b\Big)_{i}\Big|\leq\Big\|\frac{\beta}{T}R^{\top}e_{T}+\eta A^{\top}b\Big\|_{\infty},
\end{equation}
we obtain that
\begin{equation}\label{equ31}
\frac{\lambda a}{(a|z_{i}^{\ast}|+1)^{2}}-2\Big|\Big(\frac{\beta}{T}\tilde{R}^{\top}e_{T}+\eta \tilde{A}^{\top}b\Big)_{i}\Big|\geq0.
\end{equation}
Hence, for any $i\in \{1, 2, \cdots, r\}$,
$$\frac{\lambda a|z_{i}^{\ast}|}{(a|z_{i}^{\ast}|+1)^{2}}-2\Big(\frac{\beta}{T}\tilde{R}^{\top}e_{T}+\eta \tilde{A}^{\top}b\Big)_{i}z_{i}^{\ast}\geq0,$$
which is a contradiction with (\ref{equ27}), as claimed. This completes the proof.
\end{proof}

\subsection{Problem $(FP_{a,\lambda,\eta})$ solves problem $(FP_{a,\lambda})$ for any $\eta\rightarrow \infty$}\label{section3-3}
Let
\begin{equation}\label{equ32}
C_{\lambda}(x)_{(Ax=b)}=\frac{1}{T}\|Rx-\beta e_{T}\|_{2}^{2}+\lambda P_{a}(x)
\end{equation}
and
\begin{equation}\label{equ33}
C_{\lambda,\eta}(x)=\frac{1}{T}\|Rx-\beta e_{T}\|_{2}^{2}+\lambda P_{a}(x)+\eta\|Ax-b\|_{2}^{2}.
\end{equation}

\begin{theorem} \label{the5}
Suppose $x^{[\lambda]}$ is the unique minimizer of the problem $(FP_{a,\lambda})$. Then, for each fixed $\eta>0$, the minimizer $x^{[\lambda,\eta]}$ of the problem $(FP_{a,\lambda,\eta})$
converges to $x^{[\lambda]}$ as $\eta\rightarrow+\infty$.
\end{theorem}

\begin{proof}
Since $x^{[\lambda,\eta]}$ minimizes $C_{\lambda,\eta}$, it follows that
$$C_{\lambda,\eta}(x^{[\lambda,\eta]})\leq C_{\lambda,\eta}(x^{[\lambda]})=C_{\lambda}(x^{[\lambda]})_{(Ax^{[\lambda]}=b)}.$$
Consequently
$$\lambda P_{a}(x^{[\lambda,\eta]})\leq C_{\lambda,\eta}(x^{[\lambda,\eta]})\leq C_{\lambda}(x^{[\lambda]})_{(Ax^{[\lambda]}=b)},$$
\end{proof}
so that $P_{a}(x^{[\lambda,\eta]})$ is bounded, uniformly in $\eta$. This implies that the set $\{x^{[\lambda,\eta]}: \eta>0\}$ must have accumulation points that can be written as
$$\tilde{x}=\lim_{n\rightarrow\infty}x^{[\lambda,\eta_{n}]},$$
where $\eta_{n}\rightarrow +\infty$ as $n\rightarrow+\infty$.

On the other hand, we also have
$$\|Ax^{[\lambda,\eta]}-b\|_{2}^{2}\leq\frac{1}{\eta} C_{\lambda,\eta}(x^{[\lambda,\eta]})\leq \frac{1}{\eta}C_{\lambda}(x^{[\lambda]})_{(Ax^{[\lambda]}=b)}\rightarrow 0\ \ \mathrm{as} \ \ \eta\rightarrow+\infty.$$
This implies that any accumulation point $\tilde{x}$, of the type described above, satisfies $A\tilde{x}=b$. It follows that
\begin{eqnarray*}
C_{\lambda}(\tilde{x})_{(A\tilde{x}=b)}&=&\lim_{n\rightarrow+\infty}C_{\lambda}(x^{[\lambda,\eta_{n}]})_{(\|Ax^{[\lambda,\eta_{n}]}-b\|_{2}^{2}\leq\frac{1}{\eta_{n}}C_{\lambda}(x^{[\lambda]})_{(Ax^{[\lambda]}=b)})}\\
&\leq&\lim_{n\rightarrow+\infty}C_{\lambda,\eta_{n}}(x^{[\lambda,\eta_{n}]})\\
&\leq&\lim_{n\rightarrow+\infty}C_{\lambda}(x^{[\lambda]})_{(Ax^{[\lambda]}=b)}\\
&=&C_{\lambda}(x^{[\lambda]})_{(Ax^{[\lambda]}=b)}.
\end{eqnarray*}
Since $\tilde{x}$ is any accumulation point of $\{x^{[\lambda,\eta]}: \eta>0\}$ for $\eta\rightarrow+\infty$, and $x^{[\lambda]}$ is the unique minimizer of $C_{\lambda}(x)_{(Ax=b)}$, it follows that
$\tilde{x}=x^{[\lambda]}$. Since this is true for an arbitrary accumulation point of $\{x^{[\lambda,\eta]}: \eta>0\}$ the type described above, it follows that
$$\lim_{\eta\rightarrow+\infty}x^{[\lambda,\eta]}=x^{[\lambda]}.$$
This completes the proof.

\section{Algorithms for solving problems $(FP_{a,\lambda,\eta})$ and $(FP^{\geq}_{a,\lambda,\eta})$}\label{section4}
In this section, we develop the thresholding representation theories of the problems $(FP_{a,\lambda,\eta})$ and $(FP^{\geq}_{a,\lambda,\eta})$. Based on them, we propose the IFPT algorithm and the INFPT algorithm to solve the problems $(FP_{a,\lambda,\eta})$ and $(FP^{\geq}_{a,\lambda,\eta})$ for all $a>0$, respectively.

\subsection{IFPT algorithm for solving the problem $(FP_{a,\lambda,\eta})$}\label{subsection4-1}
In this subsection, the IFPT algorithm is proposed to solve the problem $(FP_{a,\lambda,\eta})$ for all $a>0$. Before we embark to this discussion, some crucial results need to be introduced for our later use.

Define a function of $\beta\in \mathbb{R}$ as
\begin{equation}\label{equequ34}
f_{\lambda}(\beta)=(\beta-\gamma)^{2}+\lambda\rho_{a}(\beta),
\end{equation}
and
\begin{equation}\label{equ35}
\mathrm{prox}_{\lambda}^{\rho_{a}}(\gamma)\triangleq\arg\min_{\beta\in \mathbb{R}}f_{\lambda}(\beta).
\end{equation}

\begin{lemma}\label{lem2}{\rm(see \cite{li15,xing16})}
The operator $\mathrm{prox}_{\lambda}^{\rho_{a}}$ defined in (\ref{equ35}) can be expressed as
\begin{equation}\label{equ36}
\mathrm{prox}_{\lambda}^{\rho_{a}}(\gamma)=\left\{
    \begin{array}{ll}
      g_{a,\lambda}(\gamma), & \ \ \mathrm{if} \ {|\gamma|> t_{a,\lambda}^{\ast};} \\
      0, & \ \ \mathrm{if} \ {|\gamma|\leq t_{a,\lambda}^{\ast}.}
    \end{array}
  \right.
\end{equation}
where $g_{a,\lambda}$ is defined as
\begin{equation}\label{equ37}
g_{a,\lambda}(\gamma)=sign(\gamma)\bigg(\frac{\frac{1+a|\gamma|}{3}(1+2\cos(\frac{\phi(\gamma)}{3}-\frac{\pi}{3}))-1}{a}\bigg),
\end{equation}
$$\phi(\gamma)=\arccos\Big(\frac{27\lambda a^{2}}{4(1+a|\gamma|)^{3}}-1\Big),$$
and the threshold value satisfies
\begin{equation}\label{equ38}
t_{a,\lambda}^{\ast}=\left\{
    \begin{array}{ll}
      t_{a,\lambda}^{1}, & \ \ \mathrm{if} \ {\lambda\leq \frac{1}{a^{2}};} \\
      t_{a,\lambda}^{2}, & \ \ \mathrm{if} \ {\lambda>\frac{1}{a^{2}}.}
    \end{array}
  \right.
\end{equation}
where
\begin{equation}\label{equ39}
t_{a,\lambda}^{1}=\frac{\lambda}{2}a, \ \ \ \ t_{a,\lambda}^{2}=\sqrt{\lambda}-\frac{1}{2a}.
\end{equation}
\end{lemma}

\begin{definition}\label{de1}
Let $x\in \mathbb{R}^{n}$, the iterative thresholding operator $G_{\lambda, P_{a}}$ can be defined as
$$\mathcal{G}_{\lambda, P_{a}}(x)=\Big(\mathrm{prox}_{\lambda}^{\rho_{a}}(x_{1}),\mathrm{prox}_{\lambda}^{\rho_{a}}(x_{2}),\cdots,\mathrm{prox}_{\lambda}^{\rho_{a}}(x_{n})\Big)^{\top}$$
where $\mathrm{prox}_{\lambda}^{\rho_{a}}$ is defined in Lemma \ref{lem2}.
\end{definition}

The iterative thresholding operator $\mathcal{G}_{\lambda, P_{a}}$ simply applies the operator $\mathrm{prox}_{\lambda}^{\rho_{a}}$ defined in Lemma \ref{lem2} to a vector,
and effectively shrink them towards zero. It is clear that if many of the entries of the vector $x$ are below the threshold value $t_{a,\lambda}^{\ast}$, the sparsity  of $\mathcal{G}_{\lambda, P_{a}}(x)$ may be considerably lower than the sparsity of vector $x$.

Nextly, we shall show that the optimal solution to the problem $(FP_{a,\lambda,\eta})$ can also be expressed a thresholding operation.

For any fixed positive parameters $\lambda>0$, $\varphi>0$, $\eta>0$, $a>0$ and $x,z\in \mathbb{R}^{n}$, let
\begin{equation}\label{equ40}
C_{\lambda,\eta,\varphi}(x,z)=\varphi[C_{\lambda,\eta}(x)-\frac{1}{T}\|Rx-Rz\|_{2}^{2}-\eta\|Ax-Az\|_{2}^{2}]+\|x-z\|_{2}^{2}
\end{equation}
be the surrogate function of the function $C_{\lambda,\eta}(x)$ defined in (\ref{equ33}). Clearly, $C_{\lambda,\eta,\varphi}(x,x)=\varphi C_{\lambda,\eta}(x)$.

\begin{theorem}\label{the6}
For any fixed $\lambda>0$, $\varphi>0$, $\eta>0$, $a>0$ and $z\in \mathbb{R}^{n}$, $\displaystyle\min_{x\in \mathbb{R}^{n}}C_{\lambda,\eta,\varphi}(x,z)$
equivalents to
$$\min_{x\in \mathbb{R}^{n}}\Big\{\|x-B_{\varphi}(z)\|_{F}^{2}+\lambda\varphi P_{a}(x)\Big\}$$
where $B_{\varphi}(z)=z+\frac{\varphi}{T}R^{T}(\beta e_{T}-Rz)+\varphi\eta A^{T}(b-Az)$.
\end{theorem}

\begin{proof}
By the definition, $C_{\lambda,\eta,\varphi}(x,z)$ can be rewritten as
\begin{eqnarray*}
C_{\lambda,\eta,\varphi}(x,z)&=&\|x-(z+\frac{\varphi}{T}R^{T}(\beta e_{T}-Rz)+\varphi\eta A^{T}(b-Az))\|_{2}^{2}+\lambda\varphi P_{a}(x)\\
&&+\|z\|_{2}^{2}-\|z+\frac{\varphi}{T}R^{T}(\beta e_{T}-Rz)+\varphi\eta A^{T}(b-Az)\|_{2}^{2}+\frac{\varphi}{T}\|\beta e_{T}\|_{2}^{2}\\
&&-\frac{\varphi}{T}\|Rz\|_{2}^{2}+\varphi\eta\|b\|_{2}^{2}-\varphi\eta\|Az\|_{2}^{2}\\
&=&\|x-B_{\varphi}(z)\|_{2}^{2}+\lambda\varphi P_{a}(x)+\|z\|_{2}^{2}-\|B_{\varphi}(z)\|_{2}^{2}+\frac{\varphi}{T}\|\beta e_{T}\|_{2}^{2}\\
&&-\frac{\varphi}{T}\|Rz\|_{2}^{2}+\varphi\eta\|b\|_{2}^{2}-\varphi\eta\|Az\|_{2}^{2}
\end{eqnarray*}
which implies that $\displaystyle\min_{x\in \mathbb{R}^{n}}C_{\lambda,\eta,\varphi}(x,z)$ for any fixed $\lambda>0$, $\varphi>0$, $\eta>0$, $a>0$ and $z\in \mathbb{R}^{n}$ equivalents to
$$\min_{x\in \mathbb{R}^{n}}\Big\{\|x-B_{\varphi}(z)\|_{2}^{2}+\lambda\varphi P_{a}(x)\Big\}.$$
This completes the proof.
\end{proof}

\begin{theorem}\label{the7}
For any fixed $\lambda>0$ and $0<\varphi<\frac{1}{\frac{1}{T}\|R\|_{2}^{2}+\eta\|A\|_{2}^{2}}$. If $x^{\ast}$ is the optimal solution of $\displaystyle\min_{x\in \mathbb{R}^{n}}C_{\lambda,\eta}(x)$, then $x^{\ast}$ is also the optimal solution of $\displaystyle\min_{x\in \mathbb{R}^{n}}C_{\lambda,\eta,\varphi}(x,x^{\ast})$, that is
$$C_{\lambda,\eta,\varphi}(x^{\ast},x^{\ast})\leq C_{\lambda,\eta,\varphi}(x,x^{\ast}).$$
\end{theorem}

\begin{proof}
By the definition of $C_{\lambda,\eta,\varphi}(x,z)$, we have
\begin{eqnarray*}
C_{\lambda,\mu,\varphi}(x,x^{\ast})&=&\varphi[C_{\lambda,\eta}(x)-\frac{1}{T}\|Rx-Rx^{\ast}\|_{2}^{2}-\eta\|Ax-Ax^{\ast}\|_{2}^{2}]+\|x-x^{\ast}\|_{2}^{2}\\
&=&\varphi[\frac{1}{T}\|Rx-\beta e_{T}\|_{2}^{2}+\lambda P_{a}(x)+\eta\|Ax-b\|_{2}^{2}]-\frac{\varphi}{T}\|Rx-Rx^{\ast}\|_{2}^{2}\\
&&-\varphi\eta\|Ax-Ax^{\ast}\|_{2}^{2}+\|x-x^{\ast}\|_{2}^{2}\\
&\geq&\varphi[\frac{1}{T}\|Rx-\beta e_{T}\|_{2}^{2}+\lambda P_{a}(x)+\eta\|Ax-b\|_{2}^{2}]\\
&=&\varphi C_{\lambda,\eta}(x)\\
&\geq&\varphi C_{\lambda,\eta}(x^{\ast})\\
&=&C_{\lambda,\eta,\varphi}(x^{\ast},x^{\ast}).
\end{eqnarray*}
This completes the proof.
\end{proof}

Theorem \ref{the7} told us that $x^{\ast}$ is the optimal solution to $\displaystyle\min_{x\in \mathbb{R}^{n}}C_{\lambda,\eta,\varphi}(x,z)$ with $z=x^{\ast}$, as long as, $x^{\ast}$ is the optimal solution of the problem $(FP_{a,\lambda,\eta})$. Combined with Theorem \ref{the6}, we derive the most important conclusion in this paper, which underlies the algorithm to be proposed.

\begin{theorem}\label{the8}
Let $x^{\ast}$ be the optimal solution of the problem $(FP_{a,\lambda,\eta})$. Then $x^{\ast}$ is also the optimal solution of the following minimization problem
$$\min_{x\in \mathbb{R}^{n}}\Big\{\|x-B_{\varphi}(x^{\ast})\|_{2}^{2}+\lambda\varphi P_{a}(x)\Big\}.$$
\end{theorem}

Combining Lemma \ref{lem2}, Definition \ref{de1} and Theorem \ref{the8}, the thresholding representation of the problem $(FP_{a,\lambda,\eta})$ can be immediately concluded as the following description.

\begin{corollary}\label{cor2}
Let $x^{\ast}\in \mathbb{R}^{n}$ be the optimal solution of the problem $(FP_{a,\lambda,\eta})$. Then it can be given by
\begin{equation}\label{equ41}
x^{\ast}=\mathcal{G}_{\lambda\varphi, P_{a}}(B_{\varphi}(x^{\ast}))
\end{equation}
where $\mathcal{G}_{\lambda\varphi,P_{a}}$ and $\mathrm{prox}_{\lambda\varphi}^{\rho_{a}}$ are obtained by replacing $\lambda$ with $\lambda\varphi$ in $\mathcal{G}_{\lambda,P_{a}}$ and $\mathrm{prox}_{\lambda}^{\rho_{a}}$.
\end{corollary}

With the representation (\ref{equ41}), the IFPT algorithm for solving the problem $(FP_{a,\lambda,\eta})$ can be naturally proposed as following:
\begin{equation}\label{equ42}
x^{k+1}=\displaystyle \mathcal{G}_{\lambda\varphi, P_{a}}(B_{\varphi}(x^{k}))
\end{equation}
where $B_{\varphi}(x^{k})=x^{k}+\frac{\varphi}{T}R^{\top}(\beta e_{T}-Rx^{k})+\varphi\eta A^{\top}(b-Ax^{k})$, which means that, in per iteration, every entries of vector $x^{k+1}$ satisfies
\begin{equation}\label{equ43}
x_{i}^{k+1}=\left\{
    \begin{array}{ll}
      g_{a,\lambda\varphi}(x_{i}^{k}), & \ \ \mathrm{if} \ {|x_{i}^{k}|> t_{a,\lambda\varphi}^{\ast};} \\
      0, & \ \ \mathrm{if} \ {|x_{i}^{k}|\leq t_{a,\lambda\varphi}^{\ast}.}
    \end{array}
  \right.
\end{equation}
for $i=1,\cdots,n$, where $g_{a,\lambda\varphi}$ and $t_{a,\lambda\varphi}^{\ast}$ are all defined in Lemma \ref{lem2} which obtained by replacing $\lambda$ with $\lambda\varphi$ in $g_{a,\lambda}$ and $t_{a,\lambda}^{\ast}$, respectively.

\begin{theorem} \label{the9}
(Convergence results of IFPT algorithm) Let $\{x^{k}\}$ be the sequence generated by the FP algorithm with the step size $\varphi$ satisfying
$0<\varphi<\frac{1}{\frac{1}{T}\|R\|_{2}^{2}+\eta\|A\|_{2}^{2}}$. Then
\begin{description}
  \item[$\mathrm{1)}$] The sequence $C_{\lambda,\eta}(x^{k})$ is decreasing.
  \item[$\mathrm{2)}$] $\{x^{k}\}$ is asymptotically regular, i.e., $\lim_{k\rightarrow\infty}\|x^{k+1}-x^{k}\|_{2}^{2}=0$.
  \item[$\mathrm{3)}$] Any accumulation point of $\{x^{k}\}$ is a stationary point of the problem $(FP_{a,\lambda,\eta})$.
\end{description}
\end{theorem}

\begin{proof}
1) By the proof of Theorem \ref{the8}, we have
$$C_{\lambda,\eta,\varphi}(x^{k+1}, x^{k})=\min_{x\in \mathbb{R}^{n}} C_{\lambda,\eta,\varphi}(x, x^{k}).$$
Combined with the definition of $C_{\lambda,\eta}(x)$ and $C_{\lambda,\eta,\varphi}(x,z)$, we have
\begin{eqnarray*}
C_{\lambda,\eta}(x^{k+1})&=&\frac{1}{\varphi}[C_{\lambda,\mu,\eta}(x^{k+1}, x^{k})-\|x^{k+1}-x^{k}\|_{2}^{2}]+\frac{1}{T}\|Rx^{k+1}-Rx^{k}\|_{2}^{2}\\
&&+\eta\|Ax^{k+1}-Ax^{k}\|_{2}^{2}.
\end{eqnarray*}
Since $0<\varphi<\frac{1}{\frac{1}{T}\|R\|_{2}^{2}+\eta\|A\|_{2}^{2}}$, we get
\begin{equation}\label{equ44}
\begin{array}{llll}
C_{\lambda,\eta}(x^{k+1})&=&\displaystyle\frac{1}{\varphi}[C_{\lambda,\eta,\varphi}(x^{k+1}, x^{k})-\|x^{k+1}-x^{k}\|_{2}^{2}]\\
&&\displaystyle+\frac{1}{T}\|Rx^{k+1}-Rx^{k}\|_{2}^{2}+\eta\|Ax^{k+1}-Ax^{k}\|_{2}^{2}.\\
&\leq&\displaystyle\frac{1}{\mu}[C_{\lambda,\eta,\varphi}(x^{k}, x^{k})-\|x^{k+1}-x^{k}\|_{2}^{2}]\\
&&\displaystyle+\frac{1}{T}\|Rx^{k+1}-Rx^{k}\|_{2}^{2}+\eta\|Ax^{k+1}-Ax^{k}\|_{2}^{2}.\\
&=&\displaystyle C_{\lambda,\eta}(x^{k})-\frac{1}{\varphi}\|x^{k+1}-x^{k}\|_{2}^{2}+\frac{1}{T}\|Rx^{k+1}-Rx^{k}\|_{2}^{2}\\
&&+\eta\|Ax^{k+1}-Ax^{k}\|_{2}^{2}.\\
&\leq&C_{\lambda,\eta}(x^{k}).
\end{array}
\end{equation}
That is, the sequence $\{x^{k}\}$ is a minimization sequence of function $C_{\lambda,\eta}(x)$, and
$$C_{\lambda,\eta}(x^{k+1})\leq C_{\lambda,\eta}(x^{k})$$
for all $k\geq0$.\\

2) Let $\theta=1-\frac{\varphi}{T}\|R\|_{2}^{2}+\varphi\eta\|A\|_{2}^{2}$. Then $\theta\in(0, 1)$ and
\begin{equation}\label{equ45}
\frac{\varphi}{T}\|Rx^{k+1}-Rx^{k}\|_{2}^{2}+\varphi\eta\|Ax^{k+1}-Ax^{k}\|_{2}^{2}\leq(1-\theta)\|x^{k+1}-x^{k}\|_{2}^{2}.
\end{equation}
By (\ref{equ44}), we have
\begin{equation}\label{equ46}
\begin{array}{llll}
&&\displaystyle\frac{1}{\varphi}\|x^{k+1}-x^{k}\|_{F}^{2}-\frac{1}{T}\|Rx^{k+1}-Rx^{k}\|_{2}^{2}-\eta\|Ax^{k+1}-Ax^{k}\|_{2}^{2}\\
&&\leq C_{\lambda,\eta}(x^{k})-C_{\lambda,\eta}(x^{k+1}).
\end{array}
\end{equation}
Combing (\ref{equ45}) and (\ref{equ46}), we get
\begin{eqnarray*}
\sum_{k=1}^{n}\|x^{k+1}-x^{k}\|_{2}^{2}&\leq&\frac{1}{\theta}\sum_{k=1}^{n}\|x^{k+1}-x^{k}\|_{2}^{2}-\frac{1}{\theta}\sum_{k=1}^{n}\frac{\varphi}{T}\|Rx^{k+1}-Rx^{k}\|_{2}^{2}\\
&&-\frac{1}{\theta}\sum_{k=1}^{n}\varphi\eta\|Ax^{k+1}-Ax^{k}\|_{2}^{2}\\
&\leq&\frac{\varphi}{\theta}\sum_{k=1}^{n}\{C_{\lambda,\eta}(x^{k})-C_{\lambda,\eta}(x^{k+1})\}\\
&=&\frac{\varphi}{\theta}(C_{\lambda,\eta}(x^{1})-C_{\lambda,\eta}(x^{n+1}))\\
&\leq&\frac{\varphi}{\theta}C_{\lambda,\eta}(x^{1}).
\end{eqnarray*}
Thus, the series $\sum_{k=1}^{\infty}\|x^{k+1}-x^{k}\|_{2}^{2}$ is convergent, which implies that
$$\|x^{k+1}-x^{k}\|_{2}^{2}\rightarrow 0 \ \ \mathrm{as}\ \ k\rightarrow\infty.$$

3) Let $\{x^{k_{j}}\}$ be a convergent subsequence of $\{x^{k}\}$, and the limit point denoted as $x^{\ast}$, i.e.,
\begin{equation}\label{equ47}
x^{k_{j}}\rightarrow x^{\ast}\ \ \mathrm{as}\ \ k_{j}\rightarrow \infty.
\end{equation}
From
$$\|x^{k_{j}+1}-x^{\ast}\|_{2}\leq\|x^{k_{j}+1}-x^{k_{j}}\|_{2}+\|x^{k_{j}}-x^{\ast}\|_{2}$$
and
$$\|x^{k_{j}+1}-x^{k_{j}}\|_{2}+\|x^{k_{j}}-x^{\ast}\|_{2}\rightarrow 0 \ \ \mathrm{as}\ \ k_{j}\rightarrow \infty,$$
we get
\begin{equation}\label{equ48}
x^{k_{j}+1}\rightarrow x^{\ast}\ \ \mathrm{as}\ \ k_{j}\rightarrow \infty.
\end{equation}
By iteration (\ref{equ42}), it follows that
$$x^{k_{j}+1}=\mathcal{G}_{\lambda\varphi,P_{a}}(B_{\varphi}(x^{k_{j}})),$$
and combined with Theorem \ref{the5}, we have
$$\|x^{k_{j}+1}-B_{\varphi}(x^{k_{j}})\|_{2}^{2}+\lambda\varphi P_{a}(x^{k_{j}+1})\leq\|x-B_{\varphi}(x^{k_{j}})\|_{2}^{2}+\lambda\varphi P_{a}(x).$$
Taking limit and using the continuity of the function $P_{a}$ as well as (\ref{equ47}) and (\ref{equ48}), we can immediately get that
$$\|x^{\ast}-B_{\varphi}(x^{\ast})\|_{2}^{2}+\lambda\varphi P_{a}(x^{\ast})\leq\|x-B_{\varphi}(x^{\ast})\|_{2}^{2}+\lambda\varphi P_{a}(x).$$
for any $x\in \mathbb{R}^{n}$, which implies that $x^{\ast}$ minimizes the function
\begin{equation}\label{equ49}
\|x-B_{\varphi}(x^{\ast})\|_{2}^{2}+\lambda\varphi P_{a}(x),
\end{equation}
and we can conclude that
$$x^{\ast}=\mathcal{G}_{\lambda\varphi,P_{a}}(B_{\varphi}(x^{\ast})).$$
This completes the proof.
\end{proof}

An important question we should face is that the solutions of a regularization problem depends seriously on the setting of the regularization parameter $\lambda$, and the selection of proper regularization parameters is always a very hard problem. In most and general cases, a "trial and error" method, say, the cross-validation method, is still an accepted, or even unique, choice. Nevertheless, when some prior information is known for a problem, it is realistic to set the regularization parameter more reasonably and intelligently.

To make this clear, let us suppose that the portfolio is required to be $r$-sparsity, that is, the portfolio should consist of $r$ assets. Let $x^{*}$ be the optimal solution to the regularization problem $(FP_{a,\lambda,\eta})$ and $|B_{\varphi}(x^{\ast})|_{i}$ be the $i$-th largest value among the absolute elements of $B_{\varphi}(x^{\ast})$. Without loss of generality, we set
$$|B_{\varphi}(x^{*})|_{1}\geq|B_{\varphi}(x^{*})|_{2}\geq\cdots\geq|B_{\varphi}(x^{*})|_{n}.$$
Then, the following inequalities hold:
$$|B_{\varphi}(x^{*})|_{i}>t_{a,\lambda\varphi}^{\ast}\Leftrightarrow i\in \{1,2,\cdots,r\},$$
$$|B_{\varphi}(x^{*})|_{j}\leq t_{a,\lambda\varphi}^{\ast}\Leftrightarrow j\in \{r+1, r+2,\cdots, n\},$$
where $t_{a,\lambda\varphi}^{\ast}$ is the threshold value which is defined in Lemma \ref{lem2} which obtained by replacing $\lambda$ with $\lambda\varphi$ in $t_{a,\lambda}^{\ast}$.

According to $t^{2}_{a,\lambda\varphi}\leq t^{1}_{a,\lambda\varphi}$, we have
\begin{equation}\label{equ50}
\left\{
  \begin{array}{ll}
   |B_{\varphi}(x^{\ast})|_{r}\geq t^{\ast}_{a,\lambda\varphi}\geq t_{a,\lambda\varphi}^{2}=\sqrt{\lambda\varphi}-\frac{1}{2a}; \\
   |B_{\varphi}(x^{\ast})|_{r+1}<t^{\ast}_{a,\lambda\varphi}\leq t_{a,\lambda\varphi}^{1}=\frac{\lambda\varphi}{2}a,
  \end{array}
\right.
\end{equation}
which implies
\begin{equation}\label{equ51}
\frac{2|B_{\varphi}(x^{\ast})|_{r+1}}{a\varphi}\leq\lambda\leq\frac{(2a|B_{\varphi}(x^{\ast})|_{r}+1)^{2}}{4a^{2}\varphi}.
\end{equation}
For convenience, we denote by $\lambda_{1}$ and $\lambda_{2}$ the left and the right of above inequality respectively. And a choice of $\lambda$ is
$$\lambda=\left\{
            \begin{array}{ll}
              \lambda_{1}, & \ \ {\mathrm{if}\ \lambda_{1}\leq\frac{1}{a^{2}\varphi};} \\
              \lambda_{2},  &\ \ {\mathrm{if}\ \lambda_{1}>\frac{1}{a^{2}\varphi}.}
            \end{array}
          \right.
$$
In practice, we approximate $x^{*}$ by $x^{k}$ in (\ref{equ51}), say, we can take
\begin{equation}\label{equ52}
\begin{array}{llll}
\lambda=\left\{
            \begin{array}{ll}
              \lambda_{1,k}=\frac{2|B_{\varphi}(x^{\ast})|_{r+1}}{a\varphi},  & \ \ {\mathrm{if}\ \lambda_{1,k}\leq\frac{1}{a^{2}\varphi};} \\
              \lambda_{2,k}=\frac{(2a|B_{\varphi}(x^{\ast})|_{r}+1)^{2}}{4a^{2}\varphi},  & \ \ {\mathrm{if}\ \lambda_{1,k}>\frac{1}{a^{2}\varphi}.}
            \end{array}
          \right.
\end{array}
\end{equation}
in applications.

When doing so, the IFPT algorithm will be adaptive and free from the choice of regularization parameter.

Notice that (\ref{equ52}) is valid for any $\varphi$ satisfying
$$0<\varphi<\frac{1}{\frac{1}{T}\|R\|_{2}^{2}+\eta\|A\|_{2}^{2}}.$$
In general, we can take
$$\varphi=\frac{1-\varepsilon}{\frac{1}{T}\|R\|_{2}^{2}+\eta\|A\|_{2}^{2}}$$
with any small $\varepsilon\in(0,1)$ below.

There is one more thing needed to be mentioned that the threshold value
$$t_{a,\lambda\varphi}^{\ast}=\left\{
            \begin{array}{ll}
              \frac{\lambda\varphi}{2}a,  & \ \ {\mathrm{if}\ \lambda=\lambda_{1,k};} \\
              \sqrt{\lambda\varphi}-\frac{1}{2a},  & \ \ {\mathrm{if}\ \lambda=\lambda_{2,k}.}
            \end{array}
          \right.
$$

\begin{algorithm}[h!]
\caption{: IFPT algorithm}
\label{alg:A}
\begin{algorithmic}
\STATE {Initialize: Choose $x^{0}\in \mathbb{R}^{n}$, $\varphi=\frac{1-\varepsilon}{\frac{1}{T}\|R\|_{2}^{2}+\eta\|A\|_{2}^{2}}$ and $a=a_{0}$ ($a_{0}$ is a given positive number);}
\STATE {\textbf{while} not converged \textbf{do}}
\STATE \ \ \ \ \ \ \ {$B_{\varphi}(x^{k})=x^{k}+\frac{\varphi}{T}R^{\top}(\beta e_{T}-Rx^{k})+\varphi\eta A^{\top}(b-Ax^{k})$;}
\STATE \ \ \ \ \ \ \ {$\lambda_{1,k}=\frac{2|B_{\varphi}(x^{k})|_{r+1}}{a\varphi}$; $\lambda_{2,k}=\frac{(2a|B_{\varphi}(x^{k})|_{r}+1)^{2}}{4a^{2}\varphi}$;}
\STATE \ \ \ \ \ \ \ {if\ $\lambda_{1,k}\leq\frac{1}{a^{2}\varphi}$\ then}
\STATE \ \ \ \ \ \ \ \ \ \ \ \ {$\lambda=\lambda_{1,k}$; $t_{a,\lambda\varphi}^{\ast}=\frac{\lambda\varphi}{2}a$}
\STATE \ \ \ \ \ \ \ \ \ \ \ \ {for\ $i=1:n$}
\STATE \ \ \ \ \ \ \ \ \ \ \ \ \ \ \ {1.\ $|B_{\varphi}(x^{k})_{i}|>t_{a,\lambda\varphi}^{\ast}$, then $x^{k+1}_{i}=g_{\lambda\varphi}(B_{\varphi}(x^{k})_{i})$;}
\STATE \ \ \ \ \ \ \ \ \ \ \ \ \ \ \ {2.\ $|B_{\varphi}(x^{k})_{i}|\leq t_{a,\lambda\varphi}^{\ast}$, then $x^{k+1}_{i}=0$;}
\STATE \ \ \ \ \ \ \ {else}
\STATE \ \ \ \ \ \ \ \ \ \ \ \ {$\lambda=\lambda_{2,k}$; $t_{a,\lambda\varphi}^{\ast}=\sqrt{\lambda\varphi}-\frac{1}{2a}$}
\STATE \ \ \ \ \ \ \ \ \ \ \ \ {for\ $i=1:n$}
\STATE \ \ \ \ \ \ \ \ \ \ \ \ \ \ \ {1.\ $|B_{\varphi}(x^{k})_{i}|>t_{a,\lambda\varphi}^{\ast}$, then $x^{k+1}_{i}=g_{\lambda\varphi}(B_{\varphi}(x^{k})_{i})$;}
\STATE \ \ \ \ \ \ \ \ \ \ \ \ \ \ \ {2.\ $|B_{\varphi}(x^{k})_{i}|\leq t_{a,\lambda\varphi}^{\ast}$, then $x^{k+1}_{i}=0$;}
\STATE \ \ \ \ \ \ \ {end}
\STATE \ \ \ \ \ \ \ {$k\rightarrow k+1$}
\STATE{\textbf{end while}}
\STATE{\textbf{return}: $x^{k+1}$}
\end{algorithmic}
\end{algorithm}

%

\subsection{INFPT algorithm for solving $(FP_{a,\lambda,\eta}^{\geq})$}\label{subsection4-2}
Inspired by the IFPT algorithm given in subsection \ref{subsection4-1}, we propose the INFPT algorithm to solve the problem $(FP_{a,\lambda,\eta}^{\geq})$ for all $a>0$.

\begin{definition}\label{de2}
Given any vector $v\in \mathbb{R}^{n}$, define the projection map on $\mathbb{R}^{n}_{+}$ by
\begin{equation}\label{equ53}
\mathcal{P}_{+}(v)\triangleq\arg\min_{\vartheta\in \mathbb{R}^{n}}\{\|\vartheta-v\|_{2}^{2}: \ \ \vartheta\geq 0\}=\max\{0, v\}.
\end{equation}
\end{definition}

\begin{theorem}\label{the10}
Let $v\in \mathbb{R}^{n}$. Then
\begin{equation}\label{equ54}
\mathcal{G}_{\lambda, P_{a}}(\mathcal{P}_{+}(v))\triangleq\arg\min_{x\in \mathbb{R}^{n}}\Big\{\|x-v\|_{2}^{2}+\lambda P_{a}(x): \ x\geq0\Big\}
\end{equation}
where $\mathcal{G}_{\lambda, P_{a}}$ is defined in Definition \ref{de1} and $\mathcal{P}_{+}$ is defined in Definition \ref{de2}.
\end{theorem}

\begin{proof}
Given any vector $v\in \mathbb{R}^{n}$, let us introduce the following notations
$$x_{+}=x_{\mathcal{I}^{+}}\ \ \mathrm{and}\ \ x_{-}=x_{\mathcal{I}^{-}},$$
where
$$\mathcal{I}^{+}=\{i\ |\ \ i\in(1,2,\cdots,n), \ \ x_{i}\geq0\}$$
and
$$\mathcal{I}^{-}=\{i\ |\ \ i\in(1,2,\cdots,n), \ \ x_{i}<0\}.$$
Observe that the following relations hold
\begin{description}
  \item[(\romannumeral1)] $\|x\|_{2}^{2}=\|x_{+}\|_{2}^{2}+\|x_{-}\|_{2}^{2}$;
  \item[(\romannumeral2)] $\|(x-v)_{+}\|_{2}^{2}+\|x_{-}\|_{2}^{2}=\|x-\mathcal{P}_{+}(v)\|_{2}^{2}$;
  \item[(\romannumeral3)] $\|x_{-}\|_{2}^{2}=0\Leftrightarrow x_{i}=0\ \ \ \forall i\in \mathcal{I}^{-}$,
\end{description}
where the second relation follows from relation (i) and the fact that $(\mathcal{P}_{+}(v))_{i}=v_{i}$ for any $i\in \mathcal{I}^{+}$ and
$(\mathcal{P}_{+}(v))_{i}=0$ for any $i\in \mathcal{I}^{-}$. By these facts, we can get that
$$\bar{x}=\mathcal{G}_{\lambda, P_{a}}(\mathcal{P}_{+}(v))$$
if and only if
\begin{eqnarray*}
\bar{x}&=&\arg\min_{x\in \mathbb{R}^{n}}\Big\{\|x-v\|_{2}^{2}+\lambda P_{a}(x): \ x\geq0\Big\}\\
&=&\arg\min_{x\in \mathbb{R}^{n}}\Big\{\|(x-v)_{+}\|_{2}^{2}+\|(x-v)_{-}\|_{2}^{2}+\lambda P_{a}(x): \ x\geq0\Big\}\\
&=&\arg\min_{x\in \mathbb{R}^{n}}\Big\{\Big(\|(x-v)_{+}\|_{2}^{2}+\|x_{-}\|_{2}^{2}-2\sum_{i\in \mathcal{I}^{-}}x_{i}v_{i}\Big)+\lambda P_{a}(x): \ x\geq0\Big\}\\
&=&\arg\min_{x\in \mathbb{R}^{n}}\Big\{\|(x-v)_{+}\|_{2}^{2}+\lambda P_{a}(x): \ x_{i}=0\ \ \forall i\in \mathcal{I}^{-}, \ x\geq0\Big\}\\
&=&\arg\min_{x\in \mathbb{R}^{n}}\Big\{\|(x-v)_{+}\|_{2}^{2}+\lambda P_{a}(x): \ \|x_{-}\|_{2}^{2}=0\Big\}\\
&=&\arg\min_{x\in \mathbb{R}^{n}}\Big\{\Big(\|(x-v)_{+}\|_{2}^{2}+\|x_{-}\|_{2}^{2}\Big)+\lambda P_{a}(x)\Big\}\\
&=&\arg\min_{x\in \mathbb{R}^{n}}\Big\{\|x-\mathcal{P}_{+}(v)\|_{2}^{2}+\lambda P_{a}(x)\Big\}\\
&=&\mathcal{G}_{\lambda,P_{a}}(\mathcal{P}_{+}(v)).
\end{eqnarray*}
This completes the proof.
\end{proof}

By Theorem \ref{the10} and inspired by iteration (\ref{equ42}), the  procedure of the INFPT algorithm for solving the regularization problem $(FP_{a,\lambda,\eta}^{\geq})$ can be inductively defined as
\begin{equation}\label{equ55}
x^{k+1}=\mathcal{G}_{\lambda, P_{a}}(\mathcal{P}_{+}(B_{\varphi}(x^{k})))
\end{equation}
where $B_{\varphi}(x^{k})=x^{k}+\frac{\varphi}{T}R^{\top}(\beta e_{T}-Rx^{k})+\varphi\eta A^{\top}(b-Ax^{k})$.

The difference between INFPT algorithm and IFPT algorithm is that the operator $\mathcal{G}_{\lambda, P_{a}}(\mathcal{P}_{+}(\cdot))$ acts only on the nonnegative part of the real line since all the $\mathcal{P}_{+}(\cdot)$ are nonnegative.

In addition, the regularization parameter $\lambda$ in INFPT algorithm can be selected as similarly as the IFPT algorithm. When doing so, the INFPT algorithm will also be adaptive and free from the choice of the regularization parameter.

Similarly, we suppose that the nonnegative vector $x_{+}^{\ast}$ of sparsity $\tilde{r}$ is the optimal solution to the regularization problem $(FP_{a,\lambda,\eta}^{\geq})$, and $|\mathcal{P}_{+}(B_{\varphi}(x^{k}))|_{i}$ be the $i$-th largest value among the elements of the nonnegative vector $\mathcal{P}_{+}(B_{\varphi}(x^{k}))$. Without loss of generality, we set
\begin{equation}\label{equ56}
|\mathcal{P}_{+}(B_{\varphi}(x^{k}))|_{1}\geq|\mathcal{P}_{+}(B_{\varphi}(x^{k}))|_{2}\geq\cdots\geq|\mathcal{P}_{+}(B_{\varphi}(x^{k}))|_{n}\geq0.
\end{equation}
Then the following inequalities hold:
$$|\mathcal{P}_{+}(B_{\varphi}(x^{k}))|_{i}>t_{a,\lambda\varphi}^{\ast}\Leftrightarrow i\in \{1,2,\cdots,\tilde{r}\},$$
$$|\mathcal{P}_{+}(B_{\varphi}(x^{k}))|_{j}\leq t_{a,\lambda\varphi}^{\ast}\Leftrightarrow j\in \{\tilde{r}+1, \tilde{r}+2,\cdots, n\}.$$
In accordance with the selection of the regularization parameter in IFPT algorithm, the optimal regularization parameter for the INFPT algorithm can be selected as
\begin{equation}\label{equ57}
\begin{array}{llll}
\lambda=\left\{
            \begin{array}{ll}
              \lambda_{1,k}=\frac{2(\mathcal{P}_{+}(B_{\varphi}(x^{k})))_{\tilde{r}+1}}{a\varphi},  & \ \ {\mathrm{if}\ \lambda_{1,k}\leq\frac{1}{a^{2}\varphi};} \\
              \lambda_{2,k}=\frac{(2a(\mathcal{P}_{+}(B_{\varphi}(x^{k})))_{\tilde{r}}+1)^{2}}{4a^{2}\varphi},  & \ \ {\mathrm{if}\ \lambda_{1,k}>\frac{1}{a^{2}\varphi}.}
            \end{array}
          \right.
\end{array}
\end{equation}
We also take
$$\varphi=\frac{1-\varepsilon}{\frac{1}{T}\|R\|_{2}^{2}+\eta\|A\|_{2}^{2}}$$
with any small $\varepsilon\in(0,1)$ below, and
$$t_{a,\lambda\mu}^{\ast}=\left\{
            \begin{array}{ll}
              \frac{\lambda\varphi}{2}a,  & \ \ {\mathrm{if}\ \lambda=\lambda_{1,k};} \\
              \sqrt{\lambda\varphi}-\frac{1}{2a},  & \ \ {\mathrm{if}\ \lambda=\lambda_{2,k}.}
            \end{array}
          \right.
$$

\begin{algorithm}[h!]
\caption{: INFPT algorithm}
\label{alg:B}
\begin{algorithmic}
\STATE {Initialize: Choose $x^{0}\in \mathbb{R}^{n}$, $\varphi=\frac{1-\varepsilon}{\frac{1}{T}\|R\|_{2}^{2}+\eta\|A\|_{2}^{2}}$ and $a=a_{0}$ ($a_{0}$ is a given positive number);}
\STATE {\textbf{while} not converged \textbf{do}}
\STATE \ \ \ \ \ \ \ {$\mathcal{P}_{+}(B_{\varphi}(x^{k}))=\mathcal{P}_{+}(x^{k}+\frac{\varphi}{T}R^{\top}(\beta e_{T}-Rx^{k})+\varphi\eta A^{\top}(b-Ax^{k}))$;}
\STATE \ \ \ \ \ \ \ {$\lambda_{1,k}=\frac{2|\mathcal{P}_{+}(B_{\varphi}(x^{k}))|_{\tilde{r}+1}}{a\varphi}$; $\lambda_{2,k}=\frac{(2a|\mathcal{P}_{+}(B_{\varphi}(x^{k}))|_{\tilde{r}}+1)^{2}}{4a^{2}\varphi}$;}
\STATE \ \ \ \ \ \ \ {if\ $\lambda_{1,k}\leq\frac{1}{a^{2}\varphi}$\ then}
\STATE \ \ \ \ \ \ \ \ \ \ \ \ {$\lambda=\lambda_{1,k}$; $t_{a,\lambda\varphi}^{\ast}=\frac{\lambda\varphi}{2}a$}
\STATE \ \ \ \ \ \ \ \ \ \ \ \ {for\ $i=1:n$}
\STATE \ \ \ \ \ \ \ \ \ \ \ \ \ \ \ {1.\ $\mathcal{P}_{+}(B_{\varphi}(x^{k}))_{i}>t_{a,\lambda\varphi}^{\ast}$, then $x^{k+1}_{i}=g_{\lambda\varphi}(\mathcal{P}_{+}(B_{\varphi}(x^{k}))_{i})$;}
\STATE \ \ \ \ \ \ \ \ \ \ \ \ \ \ \ {2.\ $\mathcal{P}_{+}(B_{\varphi}(x^{k}))_{i}\leq t_{a,\lambda\varphi}^{\ast}$, then $x^{k+1}_{i}=0$;}
\STATE \ \ \ \ \ \ \ {else}
\STATE \ \ \ \ \ \ \ \ \ \ \ \ {$\lambda=\lambda_{2,k}$; $t_{a,\lambda\varphi}^{\ast}=\max\{\sqrt{\lambda\varphi}-\frac{1}{2a},0\}$}
\STATE \ \ \ \ \ \ \ \ \ \ \ \ {for\ $i=1:n$}
\STATE \ \ \ \ \ \ \ \ \ \ \ \ \ \ \ {1.\ $\mathcal{P}_{+}(B_{\varphi}(x^{k}))_{i}>t_{a,\lambda\varphi}^{\ast}$, then $x^{k+1}_{i}=g_{\lambda\varphi}(\mathcal{P}_{+}(B_{\varphi}(x^{k}))_{i})$;}
\STATE \ \ \ \ \ \ \ \ \ \ \ \ \ \ \ {2.\ $\mathcal{P}_{+}(B_{\varphi}(x^{k}))_{i}\leq t_{a,\lambda\varphi}^{\ast}$, then $x^{k+1}_{i}=0$;}
\STATE \ \ \ \ \ \ \ {end}
\STATE \ \ \ \ \ \ \ {$k\rightarrow k+1$}
\STATE{\textbf{end while}}
\STATE{\textbf{return}: $x^{k+1}$}
\end{algorithmic}
\end{algorithm}

\section{Numerical experiments} \label{section5}
In this section, we apply the IFPT algorithm and INFPT algorithm described above to construct the optimal (sparse) portfolios with and without short-selling constraints,
and carry out a series of simulations to evaluate their out-of-sample performance. The tests and comparisons are performed on two sets of portfolios from Fama and French web site\footnote{
http://mba.tuck.dartmouth.edu/pages/faculty/ken.french/data\underline{\hbox to 1mm{}}library.html}: the 48 industry portfolios (FF48) and 100 portfolios formed on size and book-to-market (FF100), ranging from July 1976 to June 2006. Tests use these two real market data during a period of 30 years from July 1976 to June 2006, and the time period is divided into 6 equal sub-periods.
To determine $R$ and $\mu$, we use the historical returns from July 1971 until June 1976. We then solve the optimization problems using this matrix and vector, targeting an annualized
return $\beta$, equal to the average historical return, from July 1971 until June 1976, obtained by a portfolio in which all industry sectors are given the equal weight 1/n. The performance
of each portfolio is evaluated by looking at its out-of-sample total return, $m$, its out-of-sample variance, $\sigma$, and its out-of-sample Sharpe ratio, $S=\frac{m}{\sigma}$. In
order to understand the effect of sparsity on the performance of resulting portfolios, the value of $k$ in the following tables specifies the number of assets in a portfolio. In all the experiments,
we set $a=1$.

Firstly, we present the numerical results of IFPT algorithm in FF48 problem with short-selling constraint, and compare them with those obtained with the $\ell_{1}$-norm regularization portfolio
selection model (solved by LARS algorithm \cite{bro8,efron18}). For the sake of simplicity, we renamed the $\ell_{1}$-norm regularization portfolio selection model (\ref{equ4}) solved by LARS
algorithm as La($\ell_{1}$). And then we show the performance of INFPT algorithm in finding the sparse portfolio weights in FF100 problem without short-selling constraint.

\begin{table}[h!]
\centering
\begin{tabular}{|c||l|l|l|l|l|l|l|l|}\hline
Item&\multicolumn{2}{c}{$k=6$}&\multicolumn{2}{|c}{$k=8$}&\multicolumn{2}{|c}{$k=10$}&\multicolumn{2}{|c|}{$k=12$}\\
\hline
Period&IFPT&La($\ell_{1}$)&IFPT&La($\ell_{1}$)&IFPT&La($\ell_{1}$)&IFPT&La($\ell_{1}$)\\
\hline
07/76-06/81& 9.09&  6.95& 7.09& 7.06& 9.77& 8.29& 9.81& 9.48\\
\hline
07/81-06/86& 7.35& 7.87& 6.14& 7.58& 8.20& 8.06& 8.37& 8.08\\
\hline
07/86-06/91& 3.19& 3.18& 3.49& 3.20& 3.72& 3.51& 3.90& 3.63\\
\hline
07/91-06/96& 8.36& 7.88& 8.52& 7.67& 8.86& 7.30& 9.49& 6.52\\
\hline
07/96-06/01& 3.72& 4.21& 4.17& 4.15& 4.54& 3.47& 3.46& 3.44\\
\hline
07/01-06/06& 1.61& 1.55& 1.63& 1.49& 2.48& 1.04& 2.66& 0.49 \\
\hline\hline
07/76-06/06& 23.86& 28.28& 24.33& 28.26& 28.58& 27.92& 30.06& 27.14\\
\hline
\end{tabular}
\caption{\scriptsize Comparison results of Sharpe Ratio $S$ between IFPT algorithm and La($\ell_{1}$) with $k=6,8,10,12$ in problem FF48 with short-selling constraint.}\label{table1}
\end{table}

\begin{table}[h!]
\centering
\begin{tabular}{|c||l|l|l|l|l|l|l|l|}\hline
Item&\multicolumn{2}{c}{$k=14$}&\multicolumn{2}{|c}{$k=16$}&\multicolumn{2}{|c}{$k=18$}&\multicolumn{2}{|c|}{$k=20$}\\
\hline
Period&IFPT&La($\ell_{1}$)&IFPT&La($\ell_{1}$)&IFPT&La($\ell_{1}$)&IFPT&La($\ell_{1}$)\\
\hline
07/76-06/81& 10.33& 9.98& 10.03& 8.81& 9.49& 8.61& 9.24& 8.60\\
\hline
07/81-06/86& 6.59&  6.48& 7.13& 5.89& 6.79& 5.61& 6.60& 5.58 \\
\hline
07/86-06/91& 3.93& 3.88& 3.92& 3.69& 3.88& 3.66& 3.92& 3.59\\
\hline
07/91-06/96& 9.22& 4.92& 9.36& 4.49& 9.56& 4.07& 9.39& 4.04\\
\hline
07/96-06/01& 3.75& 3.28& 3.15& 3.02& 3.47& 3.04& 3.31& 3.06\\
\hline
07/01-06/06& 2.61& 0.01& 3.00& -0.09& 2.56&  -0.09& 2.34& 0.01\\
\hline\hline
07/76-06/06& 27.27& 22.62& 29.27& 20.97& 29.83& 20.77& 29.90& 20.95\\
\hline
\end{tabular}
\caption{\scriptsize Comparison results of Sharpe Ratio $S$ between IFPT algorithm and La($\ell_{1}$) with $k=14,16,18,20$ in problem FF48 with short-selling constraint.}\label{table2}
\end{table}

\begin{table}[h!]
\centering
\begin{tabular}{|c||l|l|l|l|l|l|l|l|}\hline
Period&$k=6$&$k=8$&$k=10$&$k=12$&$k=14$&$k=16$&$k=18$&$k=20$\\
\hline
07/76-06/81& 3.41& 3.52& 3.52& 3.59& 3.60& 3.65& 3.69& 3.79\\
\hline
07/81-06/86& 4.59& 4.66& 4.68& 4.70& 4.71& 4.96& 4.99& 5.02\\
\hline
07/86-06/91& 2.14& 2.12& 2.11& 2.09& 2.07& 2.06& 2.00& 1.98 \\
\hline
07/91-06/96& 12.45& 12.62& 12.61& 12.59& 12.65& 12.70& 12.74& 12.76\\
\hline
07/96-06/01& 3.69& 4.17& 4.18& 4.17& 4.14& 4.13& 4.05& 4.01\\
\hline
07/01-06/06& 1.47& 1.38& 1.41& 1.44& 1.52& 1.71& 1.76& 2.02\\
\hline\hline
07/76-06/06& 22.12& 22.59& 22.75& 22.87& 23.05& 23.41& 23.50& 23.81\\
\hline
\end{tabular}
\caption{\scriptsize The performance of INFPT algorithm in problem FF100 without short-selling constraint with $k=6,8,10,12,14,16,18,20$.}\label{table3}
\end{table}

Tables \ref{table1} and \ref{table2} report the numerical results of IFPT algorithm and La($\ell_{1}$) in FF48 problem with short-selling constraint and vary $k$ from $6$ to $20$ with step size $2$.
The numerical results show that the performance of the IFPT algorithm is better than the performance of La($\ell_{1}$) in all periods except in the period 07/76-06/06. Table \ref{table3} reports the performance of INFPT algorithm in problem FF100 without short-selling constraint. It can be observed from Table \ref{table3} that the INFPT algorithm performs effectively in finding the sparse portfolio weights in FF100 problem without short-selling constraint.

\section{Conclusions}\label{section6}
The sparsity requirement in portfolio selection problems comes from the real world practice, where the administration of a portfolio made up of a large number of assets, possibly with very small holdings for some of them, is clearly not desirable because of the transactions costs and the complexity of management. In this paper, a continuous and non-convex sparsity promoting fraction function is studied in two sparse portfolio selection models with and without short-selling constraints in terms of theory, algorithms and computation. Firstly, we study the properties of the optimal solution to the problem $(FP_{a,\lambda,\eta})$ including the first-order and the second optimality condition and the lower and upper bound of the absolute value for its nonzero entries. Secondly, the IFPT algorithm is proposed to solve the problem $(FP_{a,\lambda,\eta})$ for all $a>0$. Moreover, we also prove that the value of the regularization parameter $\lambda>0$ can not be chosen too large. Indeed, there exists $\bar{\lambda}>0$ such that the optimal solution to the problem $(FP_{a,\lambda,\eta})$ is equal to zero for any $\lambda>\bar{\lambda}$. At last, inspired by the thresholding representation of the IFPT algorithm, the INFPT algorithm is proposed to solve the problem $(FP_{a,\lambda,\eta}^{\geq})$ for all $a>0$. Empirical results show that our methods perform effective in finding the sparse portfolio weights in FF48 and FF100 problems with and without short-selling constraints.

\begin{acknowledgements}
We would like to thank editorial and referees for their comments which help us to enrich the content and improve the presentation of the results in this paper. The work was supported by
the National Natural Science Foundations of China (11771347, 11131006, 41390450).
\end{acknowledgements}



\end{document}